\documentclass{amsart}

\usepackage[utf8]{inputenc}
\usepackage[T1]{fontenc}

\usepackage{tikz}
\usepackage{lipsum}

\usepackage{pstricks}
\usepackage{caption}
\usepackage{subcaption}
\usepackage{float}
\usepackage{enumerate}
\usepackage{paralist}
\usepackage{todonotes}
\usepackage{amsthm}
\usepackage{amsmath}
\usepackage{amssymb}
\usepackage{amsrefs}
\usepackage{mathtools}
\theoremstyle{plain}
\newtheorem{theorem}{Theorem}
\newtheorem{proposition}[theorem]{Proposition}

\newtheorem{corollary}[theorem]{Corollary}
\theoremstyle{definition}
\newtheorem{definition}[theorem]{Definition}

\theoremstyle{remark}
\newtheorem{remark}[theorem]{Remark}

\usepackage[english]{babel}

\newcommand{\mb}[1]{{\mathbb{#1}}}
\newcommand{\mc}[1]{{\mathcal{#1}}}
\newcommand{\mr}[1]{{\mathrm{#1}}}
\DeclareMathOperator{\sgn}{sgn}
\DeclareMathOperator{\lk}{lk}
\newcommand{\R}{\mathbb{R}}

\newcommand{\Of}[1]{\Omega{#1}^\uparrow}
\newcommand{\Ob}[1]{\Omega{#1}^\downarrow}

\newcommand{\psunits}[1]{\psset{xunit={#1},yunit={#1},runit={#1}}}
\def\psu{.4pt}
\def\sep{70}
\psunits{\psu}

\newcommand{\move}[5]
{%
    \newcount\width
    \width=#4
    \advance\width by 300
    \psunits{#5}
    \begin{subfigure}{0.49\textwidth} 
        \centering
        \begin{pspicture}(\width,150)
            \input{./#1.pst}
            \rput[Bl](185,75){
                \input{./Rchange.pst}
            }
            \rput[Bl](220,0){
                \input{./#2.pst}
            }
            \rput[B](185,90){#3}
        \end{pspicture}
    \end{subfigure}%
}

\title{Minimal generating sets of directed oriented Reidemeister moves}
\author{Piotr Suwara}
\address{Faculty of Mathematics, Informatics and Mechanics, University of Warsaw, Warsaw, Poland}
\curraddr{Department of Mathematics, Massachusetts Institute of Technology, Cambridge, MA 02139}
\email{suwara@mit.edu}

\date{\today}

\begin{document}

\begin{abstract}
    Polyak 
    proved that the set $\{\Omega1a,\Omega1b,\Omega2a,\Omega3a\}$
    is a~minimal generating set of oriented Reidemeister moves.
    One may distinguish between \emph{forward} and \emph{backward}
    moves,
    obtaining $32$ different types of moves, which we call
    \emph{directed oriented Reidemeister moves}.
    In this article we prove that the set of $8$ directed Polyak moves
    $\{\Of{1a},\Ob{1a},\Of{1b},\Ob{1b},\Of{2a},\Ob{2a},\Of{3a},\Ob{3a}\}$
    is a~minimal generating set of directed oriented Reidemeister moves.
    We also specialize the problem,
    introducing the notion of a~\emph{$L$-generating set}
    for a~link $L$.
    The same set is proven to be a~minimal $L$-generating set
    for any link $L$ with at least $2$ components.
    Finally, we discuss knot diagram invariants
    arising in the study of $K$-generating sets 
    for an~arbitrary knot $K$,
    emphasizing the distinction between
    \emph{ascending} and \emph{descending} moves of type $\Omega3$.
\end{abstract}

\maketitle

\section{Introduction}
\label{sec:intro}

A~knot or link in $\R^3$ can be represented by its \emph{diagram},
which is a~generic projection of the knot or link on $\R^2$,
admitting no singularities, triple points and non-transversal double points,
together with a~decoration of the double points indicating
the choice of \emph{overcrossings} and \emph{undercrossings}.
The theorem of Reidemeister \cite{R27} states that 
two diagrams represent the same link if and only if they can be connected
by a sequence of \emph{Reidemeister moves} 
of three distinct types $\Omega1, \Omega2$ and $\Omega3$
(see Figure~\ref{fig:omega}).
\begin{figure}[H]
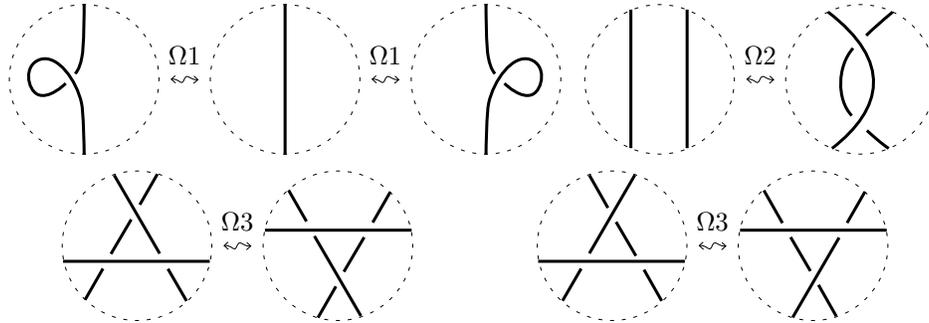
 
    \psunits{.38pt}
    \begin{subfigure}{0.59\textwidth} 
        \centering
        \begin{pspicture}(550,150)
            \input{./R1-1.pst}
            \rput[Bl](175,75){
                \input{./Rchange.pst}
            }
            \rput[B](175,90){$\Omega1$}
            \rput[Bl](200,0) {
                \input{./R1-0.pst}
            }
            \rput[Bl](375,75){
                \input{./Rchange.pst}
            }
            \rput[B](375,90){$\Omega1$}
            \rput[Bl](400,0) {
                \input{./R1-2.pst}
            }
        \end{pspicture}
    \end{subfigure}
    \begin{subfigure}{0.39\textwidth} 
        \centering
        \begin{pspicture}(350,150)
            \input{./R2-1.pst}
            \rput[Bl](175,75){
                \input{./Rchange.pst}
            }
            \rput[Bl](200,0) { \input{./R2-2.pst} }

            \rput[B](175,90){$\Omega2$}
        \end{pspicture}
    \end{subfigure}

    \vspace{5pt}
    \begin{subfigure}{0.49\textwidth} 
        \centering
        \begin{pspicture}(350,150)
            \input{./R3x-1.pst}
            \rput[Bl](175,75){
                \input{./Rchange.pst}
            }
            \rput[Bl](200,0) { \input{./R3x-2.pst} }

            \rput[B](175,90){$\Omega3$}
        \end{pspicture}
    \end{subfigure}
    \begin{subfigure}{0.49\textwidth} 
        \centering
        \begin{pspicture}(350,150)
            \input{./R3y-1.pst}
            \rput[Bl](175,75){
                \input{./Rchange.pst}
            }
            \rput[Bl](200,0) { \input{./R3y-2.pst} }

            \rput[B](175,90){$\Omega3$}
        \end{pspicture}
    \end{subfigure}
    \caption{Reidemeister moves.}
    \label{fig:omega} 
\end{figure}
Considering oriented diagrams (diagrams of oriented knots or links),
one obtains $16$ different types of \emph{oriented Reidemeister moves}
(see Figures~\ref{fig:omega1}, \ref{fig:omega2}, \ref{fig:omega3}).
Polyak proved in \cite{P10b} that the set
$\{\Omega1a,\Omega1b,\Omega2a,\Omega3a\}$
is sufficient to obtain all oriented Reidemeister moves.
Moreover, he showed that there is no smaller
(in terms of the number of elements)
set of oriented Reidemeister moves.
This finding reduces the procedure of checking whether
a~function of a~link diagram is in fact a~link invariant
to examining changes of the function under only $4$ types of moves.
A~similar study has been carried out by Kim, Joung and Lee \cite{KJL15}
for Yoshikawa moves on surface-link diagrams.
However, they have not proved that any of the generating sets
they found is minimal. 
\begin{figure}[H] 
    \foreach \a/\b in {a/b,c/d} {
        \move{R10}{R1\a}{$\Omega1\a$}{\sep}{\psu}
        \move{R10}{R1\b}{$\Omega1\b$}{\sep}{\psu}
        \vspace{10pt}
    }
    \vspace{-10pt}
    \caption{Oriented moves of type $\Omega1$.}
    \label{fig:omega1} 
\end{figure}
\begin{figure}[H] 
    \foreach \a/\b in {a/b,c/d} {
        \move{R2\a-1}{R2\a-2}{$\Omega2\a$}{\sep}{\psu}
        \move{R2\b-1}{R2\b-2}{$\Omega2\b$}{\sep}{\psu}
        \vspace{10pt}
    }
    \vspace{-10pt}
    \caption{Oriented moves of type $\Omega2$.}
    \label{fig:omega2} 
\end{figure}
\begin{figure}[h] 
    \foreach \a/\b in {a/b,c/d,e/f,g/h} {
        \move{R3\a-1}{R3\a-2}{$\Omega3\a$}{\sep}{\psu}
        \move{R3\b-1}{R3\b-2}{$\Omega3\b$}{\sep}{\psu}
        \vspace{10pt}
    }
    \vspace{-10pt}
    \caption{Oriented moves of type $\Omega3$.}
    \label{fig:omega3} 
\end{figure}

We rephrase Polyak's result introducing the notion
of a~generating set of moves:
\begin{definition}
    [generating set of moves]
    A~set $A$ of moves on oriented tangle diagrams is 
    called \emph{tangle-generating}
    (shortly, \emph{generating})
    if for any two tangle diagrams $T_1, T_2$
    representing the same tangle,
    one can obtain $T_2$ from $T_1$ using moves from $A$.
\end{definition}
Tangles are more general objects than knots and links
and in particular diagrams of oriented Reidemeister moves
may be considered tangles.
The theorem of Reidemeister generalizes for tangles:
the set of all oriented Reidemeister moves is tangle-generating.
Therefore a~set $A$ is tangle-generating
if and only if every oriented Reidemeister move
(in both directions)
may be obtained using moves from $A$
.
Thus we will drop the \emph{tangle-} prefix
and call such sets \emph{generating}.
Moreover, the result of Polyak may be phrased as follows:
the set 
$\{\Omega1a,\Omega1b,\Omega2a,\Omega3a\}$
is a~minimal (with respect to size) generating subset
of oriented Reidemeister moves.

We now generalize the problem, considering 
\emph{directed oriented Reidemeister moves},
that is, distinguishing between 
\emph{forward} and \emph{backward} moves.
\begin{definition}[directed oriented Reidemeister moves]
    We will call a~Reidemeister move of type $\Omega1$ or $\Omega2$
    \emph{forward} if it increases the number of crossings
    and \emph{backward} if it decreases the number of crossings.

    For an $\Omega3$ move, 
    let us call the triangle formed by the three crossings in the $\Omega3$ move
    diagram the \emph{vanishing triangle}. 
    There is an ordering of its sides coming from the fact that they belong to
    distinct strands, and we order them bottom-middle-top.
    This ordering gives us an orientation of the vanishing triangle.
    Now let $n$ be the number of its sides on which this orientation
    coincides with the orientation of the diagram. 
    Let $q = (-1)^n$.
    Any $\Omega3$ move changes $q$ 
    since it changes $n$ by $\pm1$ or $\pm3$.
    We define forward moves to be precisely those that change $q=-1$ to $q=+1$.

    Forward moves are presented in Figures~\ref{fig:omega1}, \ref{fig:omega2}
    and \ref{fig:omega3},
    when considering them as moves from the diagram to the left
    to the diagram to the right.
    We denote forward moves using $\uparrow$
    and backward using $\downarrow$, e.g. $\Of{1a}$,
    $\Ob{2c}$ or $\Of{3h}$.
\end{definition}
These notions are motivated by the definitions of 
\emph{positive} and \emph{negative} moves
on plane curves introduced by Arnold \cite{A94},
but slightly modified, as suggested by \"Ostlund \cite{O01}.
Moreover, in Subsection \ref{subsec:mgs-3} we present
an~equivalent definition of forward and backward
moves of type $\Omega3$.

This way we obtain $32$ distinct moves.
Motivated by Polyak's work, we seek to find 
a~minimal generating subset of these.

The only known results concerning this problem
are direct consequences of results concerning
generating sets of oriented Reidemeister moves:
a~set $A$ of oriented Reidemeister moves is generating
if and only if the set of both forward and backward types of moves from $A$
is generating.
In particular, Polyak's results imply that the set
$\{\Of{1a},\Ob{1a},\Of{1b},\Ob{1b},\Of{2a},$ $\Ob{2a},\Of{3a},\Ob{3a}\}$
which we call \emph{(directed) Polyak moves}
is generating,
and every generating subset of directed oriented moves
consists of at least $4$ moves.
These results are not sharp:
potentially, there could be a~smaller generating set,
in particular a~proper subset of Polyak moves
could be generating.

We prove that this is not the case:
\begin{theorem}
    [minimal generating set]
    The set of directed Polyak moves
    $$\{\Of{1a},\Ob{1a},\Of{1b},\Ob{1b},\Of{2a},\Ob{2a},\Of{3a},\Ob{3a}\}$$
    is a~minimal generating set of oriented directed Reidemeister moves.

    More generally, any generating subset 
    of directed oriented Reidemeister moves
    must contain:
    \begin{enumerate}
        \item at least one move from each of the sets
            $\{\Of{1a},\Ob{1d}\}$, $\{\Ob{1a},\Of{1d}\}$,
            \\ $\{\Of{1b},\Ob{1c}\}$, $\{\Ob{1b},\Of{1c}\}$,
        \item at least one forward ($\Of{2}$)
            and backward ($\Ob{2}$)
            move of type $\Omega2$,
        \item at least one forward ($\Of{3}$)
            and backward ($\Ob{3}$)
            move of type $\Omega3$.
    \end{enumerate}
    \label{thm:main}
\end{theorem}

Polyak \cite{P10b} showed the existence
of $4$-element sets of oriented Reidemeister moves
which satisfy the conditions above,
but are not generating.
Therefore these conditions are not sufficient
to determine whether a~set is generating.

To prove that some set is not generating,
it suffices to prove that it is not $L$-generating for some $L$:
\begin{definition}
    [$L$-generating set]
    Let $L$ be a~link.
    A~set $A$ of moves is \emph{$L$-generating},
    if any two diagrams $L_1, L_2$ of $L$
    are connected by a~sequence of moves from $A$.
\end{definition}
Indeed, if $A$ is generating,
then it is $L$-generating for any link $L$.
To prove Theorem \ref{thm:main}
we show the following:
\begin{theorem}
    [$\Omega1$ in $L$-generating sets]
    For any link $L$, any $L$-generating
    subset of directed oriented Reidemeister moves 
    contains at least:
    \begin{itemize}
        \item $1$ move from the set $\{\Of{1a}, \Ob{1d}\}$,
        \item $1$ move from the set $\{\Ob{1a}, \Of{1d}\}$,
        \item $1$ move from the set $\{\Of{1b}, \Ob{1c}\}$,
        \item $1$ move from the set $\{\Ob{1b}, \Of{1c}\}$.
    \end{itemize}
    \label{thm:omega1}
\end{theorem}
\begin{theorem}
    [$\Omega2$ in $L$-generating sets, for non-knot $L$]
    For any~link $L$ with at least $2$ components,
    any $L$-generating 
    subset of directed oriented Reidemeister moves contains at least
    $1$ move of type $\Of{2}$ and $1$ move of type $\Ob{2}$.

    \label{thm:omega2}
\end{theorem}
\begin{theorem}
    [$\Omega3$ in $L$-generating sets, for non-knot $L$]
    For any~link $L$ with at least $2$ components,
    any $L$-generating 
    subset of directed oriented Reidemeister moves contains at least
    $1$ move of type $\Of{3}$ and $1$ move of type $\Ob{3}$.

    \label{thm:omega3}
\end{theorem}

In fact, we answer the question
of finding a~minimal $L$-generating set for any link $L$
which is not a~knot.

It would be interesting to know if a~similar result
holds for $K$-generating sets when $K$ is a~knot.
This problem seems to be much harder to solve
and therefore we reduce it to the question
whether the set of directed Polyak moves
has $K$-generating subsets.
Theorem \ref{thm:omega1} readily implies
\begin{corollary}
    [$\Omega1$ in $L$-generating subsets of Polyak moves]
    For any link $L$,
    any $L$-generating subset of directed Polyak moves
    contains moves $\Of{1a}, \Ob{1a}, \Of{1b}$ and $\Ob{1b}$.
\end{corollary}
A~similar result holds for moves of type $\Omega2$:
\begin{theorem}[$\Omega2$ in $L$-generating subsets of Polyak moves]
    \label{thm:polyak2moves}
    For any link $L$,
    any $L$-generating subset of directed Polyak moves
    contains moves
    $\Of{2a}$ and $\Ob{2a}$.
\end{theorem}
We also present partial results concerning moves of type $\Omega3$,
distinguishing between \emph{ascending} and \emph{descending}
moves of type $\Omega3$ (see Definition \ref{def:asc-desc}).

On the other hand, for any link $L$,
the set $\{\Omega1a, \Omega1b, \Omega2a, \Omega3a\}$
is a~minimal generating subset of (undirected) oriented Reidemeister moves.
Indeed, Hagge \cite{H06} proved that for any knot $K$
(and therefore for any link, too)
there exist two diagrams $K_1, K_2$ of $K$ such that
one cannot obtain $K_2$ from $K_1$ without using moves of type $\Omega2$,
and there exist diagrams $K_3, K_4$ of $K$
such that $K_4$ cannot be obtained from $K_3$ without using moves of type $\Omega3$.
These, together with Theorem \ref{thm:omega1},
(proof of which mirrors the proof of Lemma 3.1 from \cite{P10b}),
proves that any $L$-generating subset of oriented Reidemeister moves 
contains at least
$2$ moves of type $\Omega1$,
$1$ move of type $\Omega2$ and $1$ move of type $\Omega3$.

The article begins with the proofs of Theorems \ref{thm:omega1}, \ref{thm:omega2}
and \ref{thm:omega3} in Section \ref{sec:mgs},
from which Theorem \ref{thm:main} follows.
The key ingredient to the proof of Theorem \ref{thm:omega3}
is the introduction of the invariants $CI$ and $OCI$,
which are thoroughly studied.
In Section \ref{sec:Polyak} we study knot diagram invariants
and their changes under Polyak moves,
emphasizing the difference between ascending and descending
moves of type $\Omega3$.
An invariant $HNP$,
which is a~special case of an~invariant defined by Hass and Nowik,
\cite{HN08} is introduced and discussed.
Moreover, families of invariants defined by \"Ostlund \cite{O01},
distinguishing between ascending and descending moves,
are briefly recalled.

The author would like to thank his advisor,
Maciej Borodzik,
for his insight and patience.

\section{Minimal Generating Sets}
\label{sec:mgs}
In this section we prove Theorem \ref{thm:main}
by proving Theorems \ref{thm:omega1}, \ref{thm:omega2},
\ref{thm:omega3}.

\subsection{$\Omega1$ and $\Omega2$ moves}
\label{subsec:mgs-12}

The following proof mirrors the proof of Lemma 3.1 from \cite{P10b}.
\begin{proof}[Proof of Theorem~\ref{thm:omega1}]
    The writhe $n$ and the winding number $c$
    of a~link diagram do not change under Reidemeister moves
    of type $\Omega2$ and $\Omega3$.
    Consider their sum $w_+=n+c$ 
    and difference $w_-=n-c$.
    \begin{table}[H]
        \centering
        \begin{tabular}{|l|c|c|c|c|}
            \hline
            Invariant & $\Of{1a}$ & $\Of{1b}$ & $\Of{1c}$ & $\Of{1d}$ \\
            \hline
            $n=$writhe       & $+1$      & $+1$      & $-1$      & $-1$      \\
            $c=$winding number       & $-1$      & $+1$      & $-1$      & $+1$      \\
            \hline
            $w_+ = n+c $ & $0$ & $+2$      & $-2$      & $0$       \\
            $w_- = n-c $ & $+2$& $0$       & $0$       & $-2$      \\
            \hline
        \end{tabular}
        \caption{Changes of the writhe and the winding number with respect to 
        Polyak moves.}
        \label{tab:changes-writhe-winding}
    \end{table}
    Notice $w_+$ increases only under $\Of{1b}$ and $\Ob{1c}$ moves.
    Consider a~diagram $D$ of a~link $L$ 
    and a~diagram $D'$ obtained from $D$
    by an~$\Of{1b}$ move. 
    Then ${w_+(D')-w_+(D) = 2}$, so any set of Reidemeister moves
    which transforms $D$ into $D'$ has to include at least one of the moves
    $\Of{1b}$ or $\Ob{1c}$.
    Therefore any $L$-generating set of moves contains one of these.
    Moreover ${w_+(D) - w_+(D') = -2}$, 
    and therefore any $L$-generating set of moves contains 
    at least one of the moves $\{\Ob{1b},\Of{1c}\}$.

    A~similar argument for $w_-$ shows that any $L$-generating set of moves
    contains at least one move from $\{\Of{1a}, \Ob{1d}\}$
    and from $\{\Ob{1a},\Of{1d}\}$.
\end{proof}

\begin{proof}[Proof of Theorem~\ref{thm:omega2}]

    $\Omega1$ and $\Omega3$ moves preserve the number of crossings
    between different components of the link diagram.
    The same is true for $\Omega2$ moves 
    between strands of the same link component.

    On the other hand, any $\Of{2}$ move 
    between strands belonging to different components of the link
    diagram creates $2$ such crossings
    and any $\Ob{2}$ move between strands of distinct components
    cancels $2$ such crossings.
    Let $L$ be a~link with at least $2$ components (i.e. not a~knot).
    Since any diagram of such link $L$ admits an~$\Of{2}$ move,
    therefore any $L$-generating set contains
    a~move of type $\Of{2}$ and a~move of type $\Ob{2}$.
\end{proof}

\subsection{$\Omega3$ moves}
\label{subsec:mgs-3}

\begin{definition}
    Denote by $\mc{C}_d(D)$ the set of crossings 
    of different components of a~diagram $D$.

    For $p \notin \gamma(S^1)$,
    denote by $\mr{Ind}_\gamma(p)$ the index of a~point $p\in \R^2$
    with respect to a~curve $\gamma: S^1 \to \R^2$.

    For $p \in \mc{C}(D)$,
    denote by $\sgn(p) \in \{-1,+1\}$
    the \emph{sign} of the crossing $p$.

    By a~\emph{changing disc} of a~(oriented, directed oriented) 
    Reidemeister move we mean the disc in the plane the move takes place in,
    as depicted in Figures~\ref{fig:omega1}, \ref{fig:omega2}, 
    \ref{fig:omega3} above.
\end{definition}

\begin{definition}[crossing index of a~diagram]
    Let $D$ be a~diagram of a~$3$-component link.
    For each crossing $p \in \mc{C}_d(D)$, 
    define its \emph{crossing index} as 
    $$CI(p) = \sgn(p) \cdot \mr{Ind}_\gamma(p),$$
    where $\gamma$ is 

    the component of the link diagram
    that does not pass through $p$.

    Now set the \emph{crossing index} of $D$ to be
    $$CI(D) =\ \sum_{\mathclap{p \in \mc{C}_d(D)}}\ CI(p).$$

    Finally, let $D$ be a~diagram of any $n$-component link,
    where $n \neq 3$.
    Let $D_1, \ldots, D_n$ denote the components of $D$.
    We define the crossing index of $D$ to be
    $$CI(D) =
    \quad \sum_{\mathclap{1 \leq i < j < k \leq n}}
    \quad CI(D|_{i,j,k}),$$
    where $D|_{i,j,k}$ denotes a~diagram obtained from $D$
    by forgetting all components other than $D_i, D_j$ and $D_k$.
\end{definition}

\begin{remark}
    This invariant is a~variation of Vassiliev's 
    index-type invariants of ornaments \cite{V94}.
\end{remark}

We may give an~equivalent, more direct definition of $CI$.

\begin{definition}
    [overcrossing and undercrossing curve]
    Let $D$ be a~diagram and $p$ be one of its crossings.
    Denote by $\gamma_p^o$ the curve of the diagram
    passing through $p$ as an~overcrossing.
    Denote by $\gamma_p^u$ the curve of the diagram
    passing through $p$ as an~undercrossing.
\end{definition}

\begin{proposition}
    [alternative description of $CI$]
    Let $D$ be a~$n$-component link diagram 
    and $\gamma_1, \ldots, \gamma_n$
    be the curves of the components of its diagram.
    Then
    $$CI(D) = \sum_{p \in \mc{C}_d(D)} 
    \quad \ 
    \sum_{
        \mathclap{\substack{
            1 \leq i \leq n,\\
            \gamma_i \neq \gamma_p^o, \gamma_i \neq \gamma_p^u
        }}
    }
    \ 
    \sgn(p) \cdot \mr{Ind}_{\gamma_i}(p).$$
    \begin{proof}
        For $n=3$, the formula coincides with the definition of $CI$
        for $3$-component links.
        Using this fact,
        by the definition of $CI$ for arbitrary $n$ we obtain:
        \begin{align*}
            CI(D)
            &= \quad
            \sum_{\mathclap{1 \leq i < j < k \leq n}}
            \quad 
            \left(
            \quad \ 
            \sum_{\mathclap{p \in \mc{C}_d(D|_{i,j})}}
            \ \sgn(p) \mr{Ind}_{\gamma_k}(p)
            + \sum_{\mathclap{p \in \mc{C}_d(D|_{j,k})}}
            \ \sgn(p) \mr{Ind}_{\gamma_i}(p)
            + \sum_{\mathclap{p \in \mc{C}_d(D|_{i,k})}}
            \ \sgn(p) \mr{Ind}_{\gamma_j}(p)
            \right)
            \\&
            = \quad
            \sum_{\mathclap{\substack{
                1 \leq i < j \leq n, \\
                1 \leq k \leq n, \\
                k \neq i, k \neq j
            }}}
            \qquad \ \ 
            \sum_{\mathclap{p \in \mc{C}_d(D|_{i,j})}}
            \ 
            \sgn(p) \mr{Ind}_{\gamma_k}(p)
            \\
            &=
            \ 
            \sum_{\mathclap{p \in \mc{C}_d(D)}}
            \qquad
            \sum_{
                \mathclap{\substack{
                    1 \leq k \leq n,\\
                    \gamma_k \neq \gamma_p^o, \gamma_k \neq \gamma_p^u
                }}
            } \ 
            \sgn(p) \cdot \mr{Ind}_{\gamma_k}(p).
        \end{align*}
        where $D|_{i,j}$ denotes the diagram obtained from $D$
        by forgetting all components other than $D_i$ and $D_j$.
        This finishes the proof.
    \end{proof}
\end{proposition}

\begin{figure}[h]
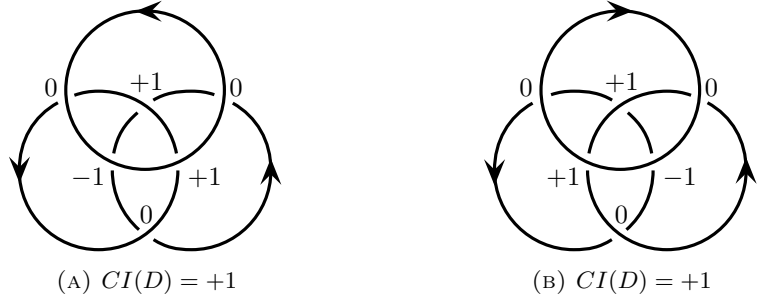

    \begin{subfigure}{0.49\textwidth}
        \centering
        \begin{pspicture}(260,240)
            \input{./3unlink-1.pst}
            \rput(40,155){$0$}
            \rput(215,155){$0$}
            \rput(130,35){$0$}
            \rput(130,160){$+1$}
            \rput(75,70){$-1$}
            \rput(185,70){$+1$}
        \end{pspicture}
        \caption{$CI(D) = +1$}
        \label{fig:omega3possible}
    \end{subfigure}
    \begin{subfigure}{0.49\textwidth}
        \centering
        \begin{pspicture}(260,240)
            \input{./3unlink-2.pst}
            \rput(40,155){$0$}
            \rput(215,155){$0$}
            \rput(130,35){$0$}
            \rput(130,160){$+1$}
            \rput(75,70){$+1$}
            \rput(185,70){$-1$}
        \end{pspicture}
        \caption{$CI(D) = +1$}
    \end{subfigure}

    \caption{Calculations of $CI(p)$, $CI(D)$ for example diagrams
        of a~$3$-component unlink.
    Crossing indices $CI(p)$ are shown next to the crossings.}
    \label{fig:ci-examples}
\end{figure}

\begin{proposition}
    The quantity $CI$ is invariant under moves of type $\Omega1$ and $\Omega2$, 
    and under moves of type $\Omega3$ which involve at least two strands
    of the same component.
    It increases by $1$ under $\Of{3}$ moves
    involving three strands of different components.

    \begin{proof}
        It follows from the construction of $CI$
        for arbitrary link diagram
        that it is sufficient to prove the claim
        for diagrams of $3$-component links.
        Therefore we assume that $D$ is a~diagram of a~$3$-component link.

        Any Reidemeister move does not change indices of points
        outside the changing disc. 
        It also does not change signs of crossings
        outside the changing disc.
        Therefore it does not change $CI(p)$ for any crossing $p$ outside
        the changing disc. 
        It suffices to check how these moves change $CI(p)$ for the crossings
        inside the changing disc.

        An $\Omega1$ move does not create or cancel any crossings between distinct
        components of a~link.
        Therefore $\Omega1$ moves do not change $CI(D)$.

        An $\Omega2$ move creates or cancels two crossings. 
        If two strands of the $\Omega2$ move belong 
        to the same component of the link diagram,
        then the move does not create or cancel any crossings
        between different components, so it preserves $CI(D)$.
        If two strands of the $\Omega2$ move belong 
        to different components of the link diagram,
        then both crossings that are created or cancelled are of different
        signs (one positive and one negative) and of the same index with respect
        to the third component, since they can be connected by a~curve
        that does not intersect the other component.
        Therefore the contributions of both crossings to $CI(D)$ cancel, 
        so $CI(D)$ is preserved by $\Omega2$ moves.

        An $\Omega3$ move, in general, does not change the set of crossings
        of the diagram, but only changes the placement 
        of three crossings involved.
        For a~crossing $p$ of two components involved in this move,
        its sign does not change, but its index with respect to 
        the third component may change (see Figure~\ref{fig:correspondence}).

        \begin{figure}[h]
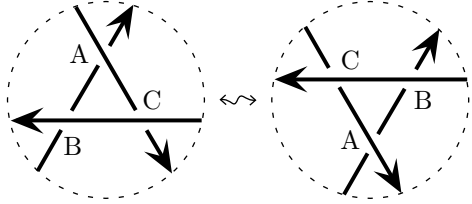

            \psunits{0.5pt}
            \centering
            \begin{pspicture}(350,150)
                \input{./R3a-1.pst}
                \rput(55,110) {A}
                \rput(50,40) {B}
                \rput(110,75) {C}
                \rput[Bl](175,75){
                    \input{./Rchange.pst}
                }
                \rput[Bl](200,0) {
                    \input{./R3a-2.pst}
                    \rput(60,45) {A}
                    \rput(115,75) {B}
                    \rput(60,105) {C}
                }
        
            \end{pspicture}
            \caption{Corresponding crossings of $\Omega3a$ move.}
            \label{fig:correspondence}
        \end{figure}

        An $\Omega3$ move that involves three strands of the same component
        does not change any crossings between different components,
        and so leaves $CI(D)$ unchanged.

        An $\Omega3$ move that involves two strands of the same component
        and one strand of any other component involves only crossings 
        between these two components.
        Since it does not change their indices with respect to the third component 
        (all points in the changing disc have the same index
        with respect to that component), 
        it does not change $CI(D)$.

        Consider an $\Omega3$ move that involves three strands 
        of different components.
        For a~crossing $p$ involved in this move,
        let $\gamma$ be the diagram component not passing through $p$,
        and $S$ be the strand taking part in the $\Omega3$ move
        contained in $\gamma$.
        The $\Omega3$ move changes the index of $p$ with respect to $\gamma$
        by $+1$ if the move shifts $p$ from the right to the left of strand $S$
        and by $-1$ if the move shifts $p$ from the left to the right of $S$.
        In the first case, $CI(p)$ changes by $+1$ if crossing $p$ is positive
        and by $-1$ if it is negative.
        In the second case, $CI(p)$ changes by $-1$ if crossing $p$ is positive
        and by $+1$ it it is negative.

        For $\Omega3$ moves involving $3$ different components,
        depicted in Figure~\ref{fig:omega3ci},
        the signs of the crossings (diagrams to the left)
        and changes of $CI(p)$ for the crossings (diagrams to the right)
        are written down.
        Summing all the changes of $CI(p) = \sgn(p)\mr{Ind}_\gamma(p)$ 
        for the three crossings of a~move,
        it follows that $CI(D)$ changes by $+1$ for moves
        of type $\Of{3}$ and by $-1$ for moves of type $\Ob{3}$.
        
        \newcommand{\cichange}[7]{%
            \newcount\width
            \width=\sep
            \advance\width by 300
            \psunits{\psu}
            \begin{subfigure}{0.49\textwidth}
                \centering
                \begin{pspicture}(\width,150)
                    \input{./R#1-1.pst}
                    \rput(54,42) {$#2$}
                    \rput(56,108) {$#3$}
                    \rput(110,72) {$#4$}
                    \rput[Bl](185,75){
                        \input{./Rright.pst}
                    }
                    \rput[Bl](220,0) {
                        \input{./R#1-2.pst}
                        \rput(36,74) {$#5$}
                        \rput(98,42) {$#6$}
                        \rput(89,102) {$#7$}

                    }
                    \rput[B](185,90){$\Of{#1}$}
                \end{pspicture}
            \end{subfigure}
        }
        \begin{figure}[h]
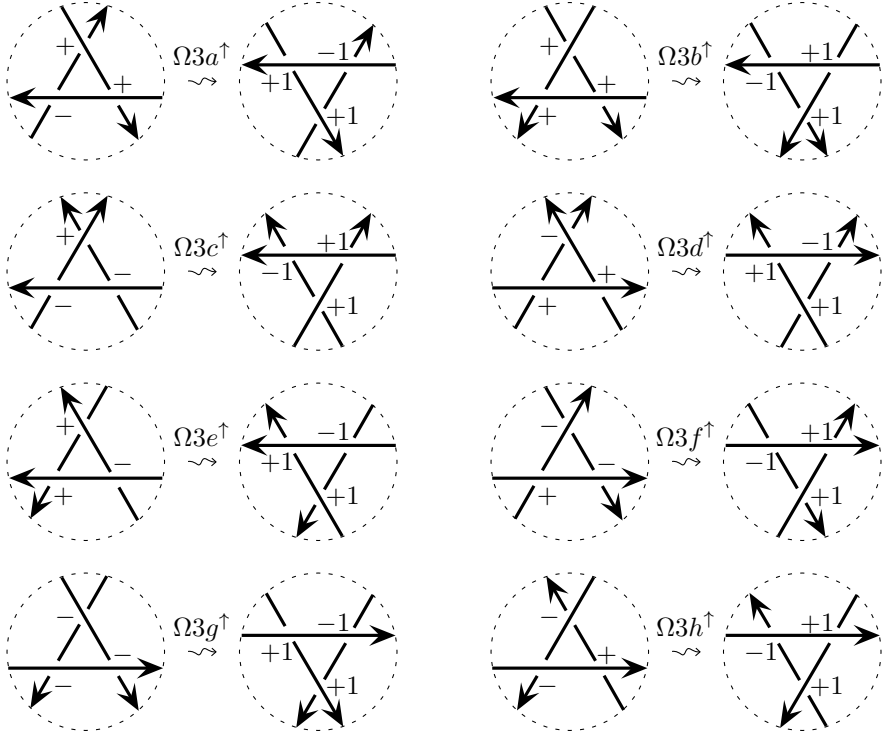
 
            \cichange{3a}{-}{+}{+}{+1}{+1}{-1}
            \cichange{3b}{+}{+}{+}{-1}{+1}{+1}

            \vspace{10pt}
            \cichange{3c}{-}{+}{-}{-1}{+1}{+1}
            \cichange{3d}{+}{-}{+}{+1}{+1}{-1}

            \vspace{10pt}
            \cichange{3e}{+}{+}{-}{+1}{+1}{-1}
            \cichange{3f}{+}{-}{-}{-1}{+1}{+1}

            \vspace{10pt}
            \cichange{3g}{-}{-}{-}{+1}{+1}{-1}
            \cichange{3h}{-}{-}{+}{-1}{+1}{+1}
            \caption{Signs (to the left) and changes of $CI(p)$ (to the right) 
            for corresponding crossings of moves of type $\Omega3$.}
            \label{fig:omega3ci}
        \end{figure}
    \end{proof}
\end{proposition}

\begin{proof}
    [Proof of Theorem~\ref{thm:omega3} for links of at least $3$ components]
    Let $L$ be a~link diagram with at least $3$ components.
    Having any diagram of $L$,
    by an~appropriate sequence of Reidemeister moves
    one can obtain a~diagram $D$ of $L$ which admits a~$\Of{3}$ move
    involving $3$ different components,
    by first making strands of $3$ different components
    bound one of the regions of the plane,
    and then making $\Omega2$ moves to obtain a~diagram
    admitting an~$\Of{3a}$ move,
    as in Figure~\ref{fig:omega3prep}.
    \begin{figure}[h]
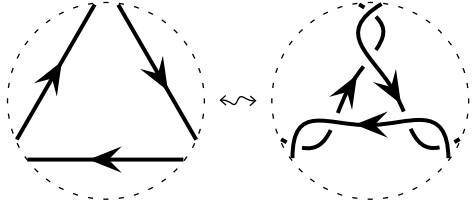

        \psunits{0.5pt}
        \centering
        \begin{pspicture}(350,150)
            \input{./R3af-prep.pst}
            \rput[Bl](175,75){
                \input{./Rchange.pst}
            }
            \rput[Bl](200,0) {
                \input{./R3af-done.pst}
            }
        \end{pspicture}
        \caption{Preparing a~diagram admitting an~$\Omega3a$ move.}
        \label{fig:omega3prep}
    \end{figure}

    The $CI$ of the diagram differs by $1$ from the $CI$ of the diagram
    obtained after performing the $\Omega3$ move.
    It follows that any $L$-generating set contains 
    at least one move of type $\Of{3}$
    and one of type $\Ob{3}$.
\end{proof}

\begin{remark}
    The above consideration yields an alternative
    characterization of forward and backward $\Omega3$ moves.
    For an~$\Omega3$ move, one may complete its diagram 
    to a~move on a~link diagram
    such that the three strands of the move diagram
    belong to different components of that link.
    The change of $CI$ of the obtained diagrams 
    due to the move
    does not depend on the chosen completion
    as we have already shown,
    so $\Of{3}$ moves may be defined to be precisely 
    the ones that increase $CI$ of a~diagram obtained this way by $1$.
\end{remark}

The invariant $CI$ is zero for any $2$-component link diagram,
so one may still ask whether both forward and backward
$\Of{3}$ moves are needed for $2$-component link diagrams.
Therefore we proceed to introduce another diagram invariant
that distinguishes forward and backward $\Omega3$ moves.

\begin{definition}[half-index]
    Let $\gamma:S^1 \to \mb{R}^2$ be an~immersed curve
    and let $p \in \gamma(S^1)$
    be a~point which is not a~double point of $\gamma$.
    Then define $\mr{hInd}_\gamma(p)$ to be the mean
    of two numbers: the index of a~point to the left of $\gamma$ close to $p$ 
    and the index of a~point to the right of $\gamma$ close to $p$.
\end{definition}

\begin{figure}[h]
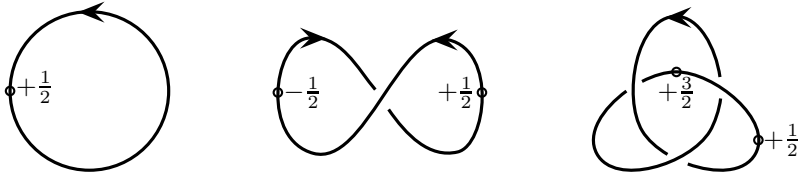

    \begin{subfigure}{0.3\textwidth}
        \centering
        \begin{pspicture}(150,160)
            \input{./unknot-0.pst}
            \pscircle(1.5,76.5){5}
            \rput(25,80){$+\frac12$}
        \end{pspicture}
    \end{subfigure}
    \begin{subfigure}{0.3\textwidth}
        \centering
        \begin{pspicture}(200,130)
            \input{./unknot-2.pst}
            \pscircle(1.5,60){5}
            \rput(25,60){$-\frac12$}
            \pscircle(194.5,60){5}
            \rput(170,60){$+\frac12$}
        \end{pspicture}

    \end{subfigure}
    \begin{subfigure}{0.3\textwidth}
        \centering
        \begin{pspicture}(160,160)
            \input{./trefoil-0.pst}
            \pscircle(157.5,30){5}
            \rput(180,30){$+\frac12$}
            \pscircle(80,94.5){5}
            \rput(80,75){$+\frac32$}
        \end{pspicture}

    \end{subfigure}

    \caption{Examples of half-indices of points with respect to 
    the underlying curves of knot diagrams.}
    \label{fig:halfindices-examples}
\end{figure}

\begin{definition}[overcrossing index]
    Let $D$ be a~diagram of a~link.
    For a~crossing $p \in \mc{C}_d(D)$, 
    we define its \emph{overcrossing index} as 
    $$OCI(p) = \sgn(p) \cdot \mr{hInd}_{\gamma_p^o}(p).$$
    Recall that $\gamma_p^o$ denotes 
    the component of the diagram that contains the overcrossing of $p$.

    Now define the \emph{overcrossing index} of $D$ to be
    $$OCI(D) =\ \sum_{\mathclap{p \in \mc{C}_d(D)}}\ 
    OCI(p).$$
\end{definition}

\begin{proposition}\label{prop:oci}
    The quantity $OCI$ in invariant under $\Omega1$ and $\Omega2$ moves,
    under $\Omega3$ moves involving $3$ strands of the same component
    and under $\Omega3$ moves involving $3$ strands of different components.

    It increases by $0$ or $1$ under an $\Of{3}$ move
    involving $2$ strands of one link component 
    and $1$ strand of another link component, 
    depending on which strands belong to the same component.
    Precisely, for such moves it increases by:
    \begin{itemize}
        \item $0$ when top and middle strands are of the same component,
        \item $0$ when middle and bottom strands are of the same component,
        \item $1$ when top and bottom strands are of the same component.
    \end{itemize}

    \begin{proof}
        As before, it suffices to check the values of $OCI(p)$ for the crossings
        in the changing discs of Reidemeister moves.

        Invariance under $\Omega1$, $\Omega2$ and 
        $\Omega3$ moves involving only one component of the link diagram
        follows from the same argument as given for $CI$.

        Invariance under $\Omega2$ moves involving
        different components of the diagram follows from
        the same argument as given for $CI$
        since both crossings involved in such move
        share the same overcrossing curve,
        and since they have opposite signs,
        their $OCI(p)$ cancel.

        $\Omega3$ moves involving strands of $3$ different components
        leave both signs and half-indices of corresponding crossings
        in the changing disc unchanged, so do not change $OCI(D)$.

        \begin{figure}[t]
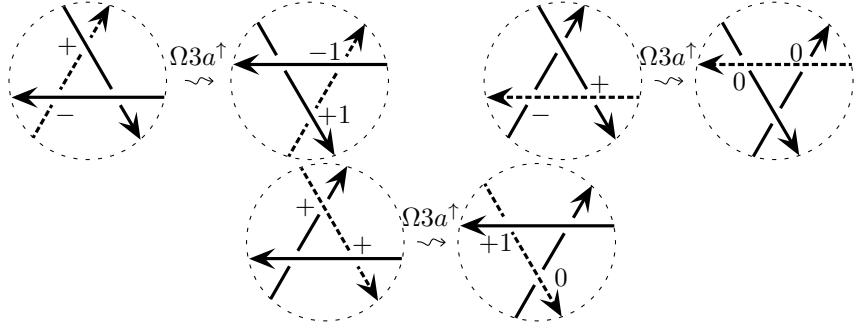


            \begin{subfigure}{0.49\textwidth} 
                \centering
                \begin{pspicture}(350,150)
                    \input{./R3a-1-b.pst}
                    \rput(56,108) {$+$}
                    \rput(54,42) {$-$}
                    \rput[Bl](180,75){
                        \input{./Rright.pst}
                    }
                    \rput[Bl](210,0) {
                        \input{./R3a-2-b.pst}
                        \rput(98,42) {$+1$}
                        \rput(89,102) {$-1$}
                    }
                    \rput[B](180,90){$\Of{3a}$}
                \end{pspicture}
            \end{subfigure}
            \begin{subfigure}{0.49\textwidth} 
                \centering
                \begin{pspicture}(350,150)
                    \input{./R3a-1-t.pst}
                    \rput(54,42) {$-$}
                    \rput(110,72) {$+$}
                    \rput[Bl](175,75){
                        \input{./Rright.pst}
                    }
                    \rput[Bl](200,0) {
                        \input{./R3a-2-t.pst}
                        \rput(96,102) {$0$}
                        \rput(42,74) {$0$}
                    }
                    \rput[B](175,90){$\Of{3a}$}
                \end{pspicture}
            \end{subfigure}

            \begin{subfigure}{0.99\textwidth} 
                \centering
                \begin{pspicture}(350,150)
                    \input{./R3a-1-m.pst}
                    \rput(56,108) {$+$}
                    \rput(110,72) {$+$}
                    \rput[Bl](175,75){
                        \input{./Rright.pst}
                    }
                    \rput[Bl](200,0) {
                        \input{./R3a-2-m.pst}
                        \rput(98,42) {$0$}
                        \rput(36,74) {$+1$}
                    }
                    \rput[B](175,90){$\Of{3a}$}
                \end{pspicture}
            \end{subfigure}
            \caption{Signs (to the left) and changes of $OCI(p)$ (to the right)
                for corresponding crossings of different components for an $\Of{3a}$
            move. The solid lines belong to one link component.}
            \label{fig:omega3a-oci}
        \end{figure}

        \begin{figure}[t]
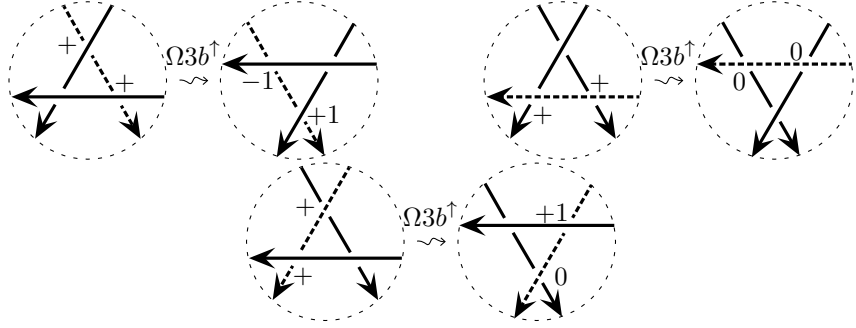

            \begin{subfigure}{0.49\textwidth} 
                \centering
                \begin{pspicture}(350,150)
                    \input{./R3b-1-b.pst}
                    \rput(56,108) {$+$}
                    \rput(110,72) {$+$}
                    \rput[Bl](175,75){
                        \input{./Rright.pst}
                    }
                    \rput[Bl](200,0) {
                        \input{./R3b-2-b.pst}
                        \rput(98,42) {$+1$}
                        \rput(36,74) {$-1$}
                    }
                    \rput[B](175,90){$\Of{3b}$}
                \end{pspicture}
            \end{subfigure}
            \begin{subfigure}{0.49\textwidth} 
                \centering
                \begin{pspicture}(350,150)
                    \input{./R3b-1-t.pst}
                    \rput(56,42) {$+$}
                    \rput(110,72) {$+$}
                    \rput[Bl](175,75){
                        \input{./Rright.pst}
                    }
                    \rput[Bl](200,0) {
                        \input{./R3b-2-t.pst}
                        \rput(42,74) {$0$}
                        \rput(96,102) {$0$}
                    }
                    \rput[B](175,90){$\Of{3b}$}
                \end{pspicture}
            \end{subfigure}

            \begin{subfigure}{\textwidth} 
                \centering
                \begin{pspicture}(350,150)
                    \input{./R3b-1-m.pst}
                    \rput(56,108) {$+$}
                    \rput(54,42) {$+$}
                    \rput[Bl](175,75){
                        \input{./Rright.pst}
                    }
                    \rput[Bl](200,0) {
                        \input{./R3b-2-m.pst}
                        \rput(98,42) {$0$}
                        \rput(89,102) {$+1$}
                    }
                    \rput[B](175,90){$\Of{3b}$}
                \end{pspicture}
            \end{subfigure}
            \caption{Signs (to the left) and changes of $OCI(p)$ (to the right)
                for corresponding crossings of different components for an $\Of{3b}$
            move. The solid lines belong to one link component.}
            \label{fig:omega3b-oci}
        \end{figure}

        An $\Omega3$ move involving strands of $2$ different components
        have $2$ crossings between different components
        in their changing discs. 
        Consider one of these crossings, $p$.
        After the $\Of{3}$ move, 
        the half-index of $p$
        with respect to $\gamma_p^o$
        stays unchanged
        if the strand $S$ not passing through $p$
        belongs to 
        $\gamma_p^u$.
        Otherwise it
        increases by $1$ if the move shifts $p$ 
        from the right to the left of strand $S$,
        or decreases by $1$ if the move shifts $p$ 
        from the left to the right of strand $S$.
        Then, to calculate the change of $OCI(p)$
        we multiply the change of this half-index
        by the sign of the crossing.

        For such $\Omega3$ move, three cases are to be considered:
        \begin{enumerate}[(a)]
            \item top and middle strands are in the same component,
            \item middle and bottom strands are in the same component,
            \item top and bottom strands are in the same component.
        \end{enumerate}

        Figure \ref{fig:omega3a-oci} summarizes signs (to the left)
        and changes of $OCI(p)$
        (to the right) in these cases for an $\Of{3a}$ move.
        Figure \ref{fig:omega3b-oci} gives the same information
        for an $\Of{3b}$ move.

        \begin{figure}[t]
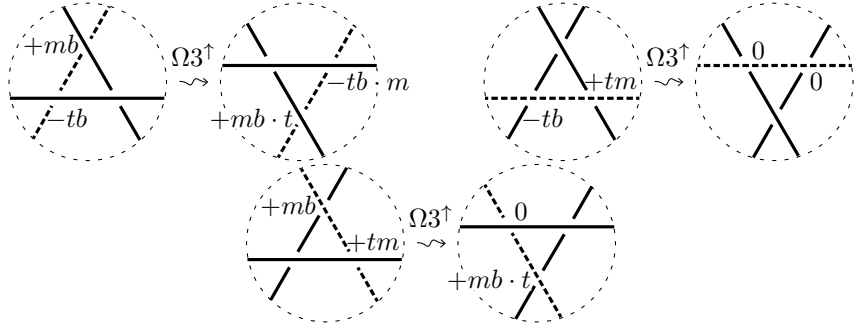

            \begin{subfigure}{0.49\textwidth} 
                \centering
                \begin{pspicture}(350,150)
                    \input{./R3x-1-b.pst}
                    \rput(40,110) {$+mb$}
                    \rput(55,40) {$-tb$}
                    \rput[Bl](175,75){
                        \input{./Rright.pst}
                    }
                    \rput[Bl](200,0) {
                        \input{./R3x-2-b.pst}
                        \psframe*[linecolor=white](0,30)(70,50)
                        \rput(30,40) {$+mb\cdot t$}
                        \psframe*[linecolor=white](110,65)(170,85)
                        \rput(140,75) {$-tb \cdot m$}
                    }
                    \rput[B](175,90){$\Of{3}$}
                \end{pspicture}
            \end{subfigure}
            \begin{subfigure}{0.49\textwidth} 
                \centering
                \begin{pspicture}(350,150)
                    \input{./R3x-1-t.pst}
                    \rput(55,40) {$-tb$}
                    \rput(120,75) {$+tm$}
                    \rput[Bl](175,75){
                        \input{./Rright.pst}
                    }
                    \rput[Bl](200,0) {
                        \input{./R3x-2-t.pst}
                        \rput(115,75) {$0$}
                        \rput(60,105) {$0$}
                    }
                    \rput[B](175,90){$\Of{3}$}
                \end{pspicture}
            \end{subfigure}

            \begin{subfigure}{\textwidth} 
                \centering
                \begin{pspicture}(350,150)
                    \input{./R3x-1-m.pst}
                    \rput(40,110) {$+mb$}
                    \rput(120,75) {$+tm$}
                    \rput[Bl](175,75){
                        \input{./Rright.pst}
                    }
                    \rput[Bl](200,0) {
                        \input{./R3x-2-m.pst}
                        \psframe*[linecolor=white](0,30)(70,50)
                        \rput(30,40) {$+mb\cdot t$}
                        \rput(60,105) {$0$}
                    }
                    \rput[B](175,90){$\Of{3}$}
                \end{pspicture}
            \end{subfigure}
            \caption{Signs (to the left) and changes of $OCI(p)$ (to the right)
                for corresponding crossings of different components for
                $\Of{3a}$, $\Of{3d}$, $\Of{3e}$, $\Of{3g}$
            moves. The solid lines belong to one link component.}
            \label{fig:omega3adeg-oci}
        \end{figure}

        \begin{figure}[t]
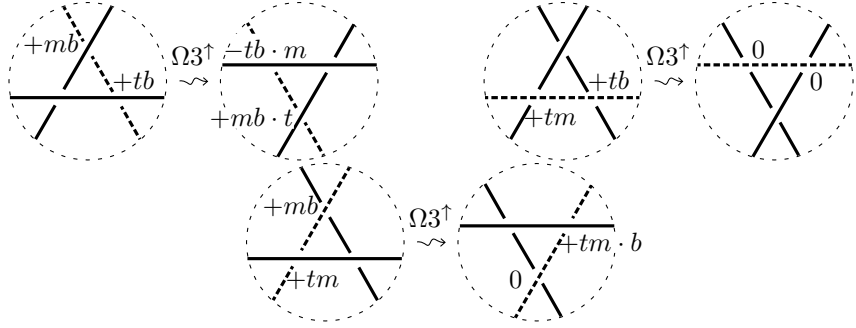

            \begin{subfigure}{0.49\textwidth} 
                \centering
                \begin{pspicture}(350,150)
                    \input{./R3y-1-b.pst}
                    \rput(40,110) {$+mb$}
                    \rput(118,75) {$+tb$}
                    \rput[Bl](175,75){
                        \input{./Rright.pst}
                    }
                    \rput[Bl](200,0) {
                        \input{./R3y-2-b.pst}
                        \psframe*[linecolor=white](0,30)(70,50)
                        \rput(30,40) {$+mb\cdot t$}
                        \psframe*[linecolor=white](13,95)(83,120)
                        \rput(43,105) {$-tb\cdot m$}
                    }
                    \rput[B](175,90){$\Of{3}$}
                \end{pspicture}
            \end{subfigure}
            \begin{subfigure}{0.49\textwidth} 
                \centering
                \begin{pspicture}(350,150)
                    \input{./R3y-1-t.pst}
                    \rput(63,40) {$+tm$}
                    \rput(118,75) {$+tb$}
                    \rput[Bl](175,75){
                        \input{./Rright.pst}
                    }
                    \rput[Bl](200,0) {
                        \input{./R3y-2-t.pst}
                        \rput(60,105) {$0$}
                        \rput(115,75) {$0$}
                    }
                    \rput[B](175,90){$\Of{3}$}
                \end{pspicture}
            \end{subfigure}

            \begin{subfigure}{\textwidth} 
                \centering
                \begin{pspicture}(350,150)
                    \input{./R3y-1-m.pst}
                    \rput(42,110) {$+mb$}
                    \rput(63,40) {$+tm$}
                    \rput[Bl](175,75){
                        \input{./Rright.pst}
                    }
                    \rput[Bl](200,0) {
                        \input{./R3y-2-m.pst}
                        \rput(55,40) {$0$}
                        \psframe*[linecolor=white](110,65)(170,85)
                        \rput(135,75) {$+tm\cdot b$}
                    }
                    \rput[B](175,90){$\Of{3}$}
                \end{pspicture}
            \end{subfigure}
            \caption{Signs (to the left) and changes of $OCI(p)$ (to the right)
                for corresponding crossings of different components for 
                $\Of{3b}$, $\Of{3c}$, $\Of{3f}$, $\Of{3h}$
            moves. The solid lines belong to one link component.}
            \label{fig:omega3bcfh-oci}
        \end{figure}

        Now, take any $\Of{3}$ move from $\Of{3a}$, $\Of{3d}$,
        $\Of{3e}$ and $\Of{3g}$. 
        The diagrams of these moves differ only by orientations
        of strands.
        Let $t$ (resp. $m$, $b$) be equal to $+1$ if the orientation
        of top (resp. middle, bottom) strand coincides 
        with the orientation of top (resp. middle, bottom) strand
        for an $\Of{3a}$ move and $-1$ otherwise.
        Figure \ref{fig:omega3adeg-oci} summarizes signs (to the left)
        and changes of $OCI(p)$
        (to the right) for the three cases of
        a~move of type $\Of{3a}$, $\Of{3d}$, $\Of{3e}$ or $\Of{3g}$.
        Analogous information is contained in Figure \ref{fig:omega3bcfh-oci}
        for moves of type $\Of{3b}$, $\Of{3c}$, $\Of{3f}$, $\Of{3h}$
        ($t, m, b$ depend of the orientations of strands
        relative to the $\Of{3b}$ move).

        Summing changes of $OCI(p)$ for crossings of these diagrams,
        it follows that in the first two cases $OCI(D)$ remains unchanged.
        In the third case, when the top and bottom strands belong to one component,
        $OCI(D)$ changes by $t\cdot m\cdot b$. 
        Checking values of $t, m, b$ for all $\Of{3}$ moves
        it follows that $tmb = 1$ for any $\Of{3}$ move.
        Indeed, each of the diagrams of moves
        $\Of{3d}, \Of{3e}, \Of{3g}$ has exactly two strands
        with orientations opposite to orientations of corresponding
        strands in $\Of{3a}$ move,
        and similar conclusion applies for $\Of{3c}, \Of{3f}, \Of{3h}$
        with respect to $\Of{3b}$.
    \end{proof}
\end{proposition}

\begin{proof}[Proof of Theorem \ref{thm:omega3}.]

    If a~link $L$ has at least $2$ components,
    a~suitable sequence of Reidemeister moves
    leads to a~diagram $D$,
    part of which looks like the left diagram
    of Figure \ref{fig:omega3prep}, 
    with the bottom and the left strand belonging to the same component
    and the strand to the right belonging to another component.
    By conducting three moves of type $\Omega2$
    as in Figure \ref{fig:omega3prep}
    we obtain a~diagram admitting
    an~$\Of{3a}$ move that increases $OCI(D)$ by $1$.
\end{proof}

\begin{remark}
    In a~similar way one can define the \emph{undercrossing index} $UCI$
    of a~diagram.
    Repeating the steps of the proof of Proposition \ref{prop:oci}
    one can show that $UCI$ changes exactly in the same way as $OCI$,
    so the difference $OCI - UCI$ is a~link invariant.
    One can directly check that the difference
    is invariant under changes of crossings,
    and is zero on an~unknot diagram.
    It follows that $OCI = UCI$.
\end{remark}

\begin{remark}
    Hayashi, Hayashi and Nowik constructed in \cite{HHN12}
    a~family of unlink diagrams $D_n$ and proved that
    the number of moves needed to separate both components of $D_n$
    is greater or equal to $(n^2+14n-13)/16$,
    and the number of moves needed to obtain a~diagram without crossings
    from $D_n$ is greater or equal to $(n^2+10n-13)/4$.
    But $OCI(D_n) = -n^2/4$ for $n$ even and
    $OCI(D_n) = -(n^2-1)/4$ for $n$ odd,
    so it follows that one needs at least $(n^2-1)/4$ moves
    (of very specific type,
    as described in Proposition \ref{prop:oci})
    to separate components of $D_n$.
\end{remark}

\section{Polyak moves}
\label{sec:Polyak}

\subsection{$\Omega2$ moves}
\label{subsec:Polyak-2}

\begin{proof}[Proof of Theorem~\ref{thm:polyak2moves}]
    Notice moves of type $\Omega1a$ and $\Omega1b$ do not change the number
    of negative crossings, $n_-$.
    This quantity is invariant under $\Omega3$ moves, too.

    On the other hand, $\Of{2a}$ increases $n_-$ by $1$,
    and $\Ob{2a}$ decreases $n_-$ by $1$.
    Therefore, having two diagrams $D_1, D_2$ of a~knot $K$,
    $D_2$ being obtained from $D_1$ by an~$\Of{2a}$ move,
    we have $n_-(D_2) - n_-(D_1) = 1$, so
    one cannot get $D_2$ from $D_1$ using directed Polyak moves without
    $\Of{2a}$ and one cannot get $D_1$ from $D_2$ using directed Polyak moves
    without $\Ob{2a}$.
\end{proof}

\subsection{Ascending and descending $\Omega3$ moves}
\label{subsec:Polyak-ad}

We recall the definition of a~diagram invariant
introduced by Hass and Nowik in \cite{HN08}.
Let $D$ be a~knot diagram and $p$ one of its crossings.
Denote by $D_p$ the link diagram obtained by smoothing
the crossing $p$ as shown in Figure~\ref{fig:smoothing}.
Let $\mc{C}_+(D)$ (resp. $\mc{C}_-(D)$) 
be the set of all positive (resp. negative) crossings of $D$.
\begin{figure}[h]
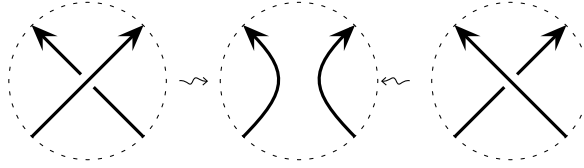

    \centering
    \begin{pspicture}(550,150)
        \input{./positive-crossing.pst}
        \rput[Bl](175,75){
            \input{./Rright.pst}
        }
        \rput[Bl](200,0){
            \input{./smooth-crossing.pst}
        }
        \rput[Br](375,75){
            \psscalebox{-1 1}{\input{./Rright.pst}}

        }
        \rput[Bl](400,0){
            \input{./negative-crossing.pst}
        }
    \end{pspicture}
    \caption{Smoothing positive and negative crossings.}
    \label{fig:smoothing}
\end{figure}

\begin{definition}\label{def:HNinvariant}
    Let $\phi$ be a~two-component link invariant
    with values in a~set $S$.
    Define a~diagram invariant
    \begin{equation}\label{eq:HNinvariant}
        I_\phi(D) =\ \sum_{\mathclap{p \in \mc{C}_+(D)}}\ X_{\phi(D_p)}
        +\ \sum_{\mathclap{p \in \mc{C}_-(D)}}\ Y_{\phi(D_p)},
    \end{equation}
    with values in 
    $\displaystyle G(S) = \bigoplus_{s \in S} (\mb{N} X_s \oplus \mb{N} Y_s)$,
    where we consider $X_s, Y_s$ to be formal variables 
    representing generators of $\displaystyle \bigoplus_{s \in S} \mb{N}^2$.
\end{definition}

We will call it the \emph{Hass--Nowik invariant}.
In their paper \cite{HN08} Hass and Nowik calculated how this invariant,
taken with $\phi = \lk$ (the linking number),
changes with respect to Reidemeister moves.

For moves we are interested in, changes of the invariant 
are summarized in the table below (following \cite{HN08}):
\begin{table}[H]
    \centering
    \begin{tabular}{|l|c|}
        \hline
        Move & Change \\
        \hline
        $\Of{1a}$ & $X_0$ \\
        $\Of{1b}$ & $X_0$ \\
        $\Of{2a}$ & $X_n + Y_{n+1}$ \\
        $\Of{3a}$ & $\pm (Y_n - Y_{n-1})$ \\
        \hline
    \end{tabular}
    \caption{Changes of $I_{\lk}$ with respect to Polyak moves.}
    \label{tab:HN-P-change}
\end{table}
Here both $n$ and $+$ or $-$ sign for $\pm$ depend on 
the part of the diagram outside the changing disc.

\begin{definition}\label{def:HNP}
    Denote by $HNP$ the diagram invariant defined as
    a~composition of $I_{\lk}$ and 
    a~semigroup homomorphism 
    $\bigoplus_{n \in \mb{Z}} (\mb{N} X_n \oplus \mb{N} Y_n) \to \mb{Z}$
    mapping $X_n \mapsto -n$, $Y_n \mapsto n-1$.
    More explicitly,
    \begin{equation}\label{eq:hnp}
        HNP(D) =\ \sum_{\mathclap{C \in \mc{C}_+(D)}}\ \lk(D_C)
        -\ \sum_{\mathclap{C \in \mc{C}_-(D)}}\ (\lk(D_C)-1)
    \end{equation}
\end{definition}

Considering the changes of $I_{\lk}$ under Polyak moves
as written in Table~\ref{tab:HN-P-change},
we notice that $HNP$ is invariant under 
$\Omega1a, \Omega1b$ and $\Omega2a$ moves
and changes by $\pm 1$ under $\Omega3a$ moves.
Carefully investigating the change of $I_{\lk}$ under $\Omega3a$ moves
we can distinguish between two different situations.

\begin{definition}[ascending and descending moves]
    \label{def:asc-desc}
    We will call an $\Omega3$ move on an oriented knot diagram 
    to be
    \emph{ascending} (resp. \emph{descending}),
    if the order of three strands involved in the move
    when traversing the knot, in the direction of orientation,
    is from bottom to top (resp. top to bottom),
    as shown (schematically) in Figure~\ref{fig:ascending} 
    (resp. Figure~\ref{fig:descending}).
\end{definition}

\begin{figure}[h]
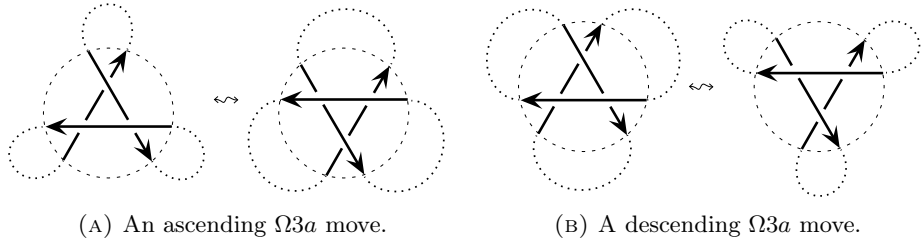

    \psunits{0.33pt}
    \begin{subfigure}{0.49\textwidth}
        \centering
        \begin{pspicture}(540,220)

            \rput[Bl](50,20) { \input{./R3a-1-ascending.pst} }
            \rput[Bl](260,110) { \input{./Rchange.pst} }

            \rput[Bl](320,20) { \input{./R3a-2-ascending.pst} }
        \end{pspicture}
        \caption{An~ascending $\Omega3a$ move.}
        \label{fig:ascending}
    \end{subfigure}
    \begin{subfigure}{0.49\textwidth}
        \centering
        \begin{pspicture}(540,220)
            \rput[Bl](50,50) { \input{./R3a-1-descending.pst} }
            \rput[Bl](260,125) { \input{./Rchange.pst} }
            \rput[Bl](320,50) { \input{./R3a-2-descending.pst} }
        \end{pspicture}
        \caption{A~descending $\Omega3a$ move.}
        \label{fig:descending}
    \end{subfigure}
    \caption{Ascending and descending $\Omega3a$ moves.}
\end{figure}

\begin{remark}
    \"Ostlund \cite{O01} calls
    forward ascending and backward descending $\Omega3$ moves 
    \emph{positive},
    and forward descending and backward ascending
    $\Omega3$ moves
    \emph{negative}.
\end{remark}

We denote an ascending or a~descending move by adding an appropriate
subscript to the move name, e.g. $\Of{3a}_{a}$ for
an ascending $\Of{3a}$ move or $\Of{3a}_{d}$ for a descending one.

\begin{proposition}\label{prop:HN-changes}
    $I_{\lk}$ changes by $Y_n-Y_{n-1}$ under an~$\Of{3a}_{a}$ move
    and by $-Y_n + Y_{n-1}$ under an~$\Of{3a}_{d}$ move,
    for some $n \in \mb{Z}$.
    \begin{proof}
        If we smooth a~diagram $D$ at crossing $p$,
        then the value of any link invariant on the smoothing
        does not depend on Reidemeister moves performed on
        the smoothed diagram $D_p$.
        What follows is that performing any Reidemeister move
        on a~knot diagram $D$ does not change either signs $\sgn(p)$
        or values of $\phi(D_p)$ for any crossing $p$ outside
        of the changing disc of this Reidemeister move.
        Therefore, in order to calculate the change of $I_{\lk}$,
        it suffices to check the values of $\phi$
        on diagrams obtained by smoothing the crossings
        involved in the move.

        An $\Of{3a}$ move does not create or cancel crossings,
        or change signs of any crossings,
        but moves them in a~particular way,
        giving a~correspondence between crossings before
        and after performing the move,
        as depicted in Figure~\ref{fig:correspondence}.
        We will distinguish these three crossings 
        by strands that pass through them: 
        top and middle, middle and bottom, or bottom and top.

        Smoothing the crossing of top and middle strand
        we obtain isotopic links before and after the $\Of{3a}$ move
        (as seen in Figure~\ref{fig:3a-topmid}).
        The same is true for the crossing of middle and bottom strand
        (Figure~\ref{fig:3a-midbot}).
        The situation is different when considering 
        top and bottom strands' crossing. 
        Smoothing before and after the $\Of{3a}$ move
        we obtain two distinct links.

        \begin{figure}[h]
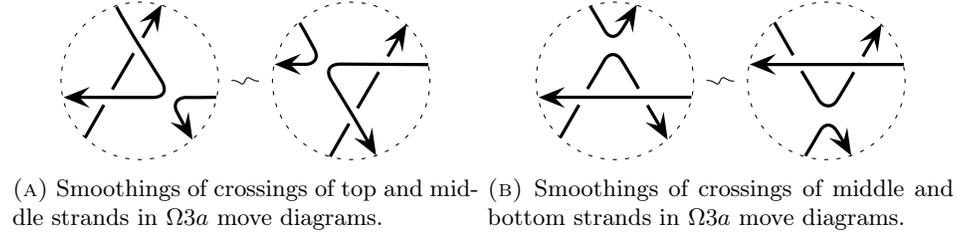

            \begin{subfigure}{0.49\textwidth}
                \centering
                \begin{pspicture}(350,150)
                    \input{./R3a-1-topmid.pst}
                    \rput[Bl](175,75) {
                        \input{./Rsim.pst}
                    }
                    \rput[Bl](200,0) {
                        \input{./R3a-2-topmid.pst}
                    }
                \end{pspicture}
                \caption{Smoothings of crossings of top and middle strands
                in $\Omega3a$ move diagrams.}
                \label{fig:3a-topmid}
            \end{subfigure}
            \begin{subfigure}{0.49\textwidth}
                \centering
                \begin{pspicture}(350,150)
                    \input{./R3a-1-midbot.pst}
                    \rput[Bl](175,75) {
                        \input{./Rsim.pst}
                    }
                    \rput[Bl](200,0) {
                        \input{./R3a-2-midbot.pst}
                    }
                \end{pspicture}
                \caption{Smoothings of crossings of middle and bottom strands
                in $\Omega3a$ move diagrams.}
                \label{fig:3a-midbot}
            \end{subfigure}
            \caption{Isotopic smoothings of corresponding crossings 
            taking part in an $\Omega3a$ move.}
        \end{figure}

        For an ascending move, the middle (straight) strand
        and the upper-right strand of the smoothing
        (as seen in Figure~\ref{fig:3a-ascending-topbot})
        belong to the same component and the lower-left strand
        belongs to the other component.
        The linking number of the smoothing,
        which is equal to some number $n$,
        increases by $1$ since the two other crossings are positive
        and while before the move (and after smoothing) 
        these were crossings between strands of one of the components, 
        after the move they become crossings between different components
        of the link diagram.
        The crossing of the top and bottom strand contributes
        $Y_n$ to $I_{\lk}$ before the move and $Y_{n+1}$ after the move.
        This, up to a~shift of $n$ by $1$, 
        proves the first part of the proposition.

        \begin{figure}[h]
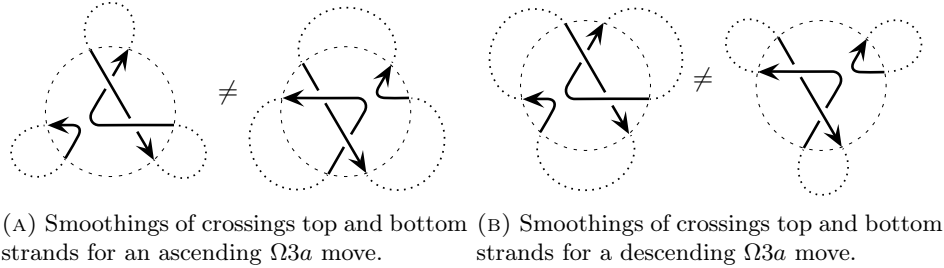

            \psunits{0.33pt}
            \begin{subfigure}{0.49\textwidth}
                \centering
                \begin{pspicture}(540,220)

                    \rput[Bl](50,20) { \input{./R3a-1-asc-topbot.pst} }
                    \rput[B](260,110) {$\neq$}

                    \rput[Bl](320,20) { \input{./R3a-2-asc-topbot.pst} }
                \end{pspicture}
                \caption{Smoothings of crossings top and bottom strands 
                for an ascending $\Omega3a$ move.}
                \label{fig:3a-ascending-topbot}
            \end{subfigure}
            \begin{subfigure}{0.49\textwidth}
                \centering
                \begin{pspicture}(540,220)
                    \rput[Bl](50,50) { \input{./R3a-1-desc-topbot.pst} }
                    \rput[B](260,125) {$\neq$}
                    \rput[Bl](320,50) { \input{./R3a-2-desc-topbot.pst} }
                \end{pspicture}
                \caption{Smoothings of crossings top and bottom strands 
                for a~descending $\Omega3a$ move.}
                \label{fig:3a-descending-topbot}
            \end{subfigure}

            \caption{Nonisotopic smoothings of corresponding crossings
            taking part in an $\Omega3a$ move.}
            \label{fig:3a-topbot}
        \end{figure}

        For a~descending move, the middle strand and the lower-left strand
        of the smoothing
        belong to one link component and the upper-right strand 
        to the other component (Figure~\ref{fig:3a-descending-topbot}).
        Similarly, in this case $2$ positive crossings 
        between these components become crossings between strands 
        of the same link component. 
        Therefore in this case the linking number of this smoothing 
        decreases after performing an $\Of{3a}$ move.
        Before this move the top and bottom strands' crossing contributes
        $Y_n$ to $I_{\lk}$ and after the move 
        it contributes $Y_{n-1}$ to $I_{\lk}$, and the proposition follows.
    \end{proof}
\end{proposition}

\begin{corollary}
    The quantity $HNP$ increases by $1$ under an $\Of{3a}_{a}$ move,
    decreases by $1$ under an $\Of{3a}_{d}$ move,
    and is invariant with respect to $\Omega1a$, $\Omega1b$ and $\Omega2a$ 
    moves.
    \begin{proof}
        It follows from evaluating changes of $I_{\lk}$ given in
        Proposition~\ref{prop:HN-changes} and in Table~\ref{tab:HN-P-change}
        via map $X_n \mapsto -n$ and $Y_n \mapsto n-1$.
    \end{proof}
\end{corollary}

This gives a~partial answer to our problem:
\begin{corollary}\label{condition-HN}
    Any knot-generating subset of
    $$\{\Of{1a},\Ob{1a},\Of{1b},\Ob{1b},\Of{2a},\Ob{2a},
    \Of{3a}_{a},\Ob{3a}_{a},\Of{3a}_{d},\Ob{3a}_{d}\}$$
    (i.e. directed Polyak moves with distinct ascending and descending moves)
    contains at least one move from the set
    $\{\Of{3a}_{a},\Ob{3a}_{d}\}$
    and one move from the set
    $\{\Ob{3a}_{a},\Of{3a}_{d}\}$.
\end{corollary}

The terms \emph{ascending} and \emph{descending} with regard to $\Omega3$ moves
are taken from the work of \"{O}stlund \cite{O01}.
In his paper, \"{O}stlund defines three families of knot diagram invariants,
namely $A_n, D_n$ for $n \geq 4$ and $W_n$ for $n \geq 3$ and $n$ odd.

He proves that
\begin{proposition}[\cite{O01}]
    $A_n$, $D_n$ and $W_n$ are invariant with respect to 
    $\Omega1$ and $\Omega2$ moves.
    Moreover, $A_n$ is invariant with respect to descending $\Omega3$ moves
    and $D_n$ is invariant with respect to ascending $\Omega3$ moves.
\end{proposition}

Then he considers the figure eight knot diagram and its inverse 

, showing that both $A_4$ and $D_4$
take different values on these two diagrams, and deduces that
\begin{theorem}
    Figure eight knot diagram cannot be transformed into its inverse
    without the use of both ascending and descending $\Omega3$ moves.
\end{theorem}

It follows that
\begin{corollary}\label{condition-Ost}
    Let $K$ be the figure eight knot.
    Any $K$-generating subset of
    $$\{\Of{1a},\Ob{1a},\Of{1b},\Ob{1b},\Of{2a},\Ob{2a},
    \Of{3a}_{a},\Ob{3a}_{a},\Of{3a}_{d},\Ob{3a}_{d}\}$$
    contains at least one move from the set
    $\{\Of{3a}_{a},\Ob{3a}_{a}\}$
    and one move from the set
    $\{\Of{3a}_{d},\Ob{3a}_{d}\}$.
\end{corollary}

Still, having both $\Of{3a}_{a}$ and $\Of{3a}_{d}$ moves
(or $\Ob{3a}_{a}$ and $\Ob{3a}_{d}$) is sufficient
to meet both this condition and 
the condition presented in Corollary~\ref{condition-HN}.
Therefore the question of necessity of containing
both $\Of{3a}$ and $\Ob{3a}$ in
$K$-generating subsets of directed Polyak moves
remains open (even in the case of the figure eight knot).

\end{document}